\documentclass[english,12pt,a4paper]{article} 

\usepackage[arrow,matrix]{xy}
\usepackage{babel}
\usepackage{amssymb}
\usepackage{amsmath}
\usepackage{stmaryrd}
\usepackage{amsthm}
\usepackage{amscd}
\usepackage{pifont}
\usepackage{textcomp}
\usepackage{bbm}
\usepackage[mathscr]{euscript}

\theoremstyle{plain}
\newtheorem{Theorem}{Theorem}[section]
\newtheorem*{MainTheorem}{Main Theorem}

\newtheorem{Proposition}[Theorem]{Proposition}

\newtheorem{Lemma}[Theorem]{Lemma}

\newtheorem{Corollary}[Theorem]{Corollary}

\theoremstyle{definition}
\newtheorem{Definition}[Theorem]{Definition}
\newtheorem{Notation}[Theorem]{Notation}
\newtheorem{Remark}[Theorem]{Remark}
\newtheorem{Setup}[Theorem]{Setup}
\newtheorem{Construction}[Theorem]{Construction}

\newcommand{\Gal}{{\rm Gal}}

\newcommand{\Hom}{{\rm Hom}}

\newcommand{\Z}{{\mathbb{Z}}}
\newcommand{\bbZ}{\Z}
\newcommand{\C}{{\mathbb{C}}}
\newcommand{\bbC}{\C}
\newcommand{\R}{{\mathbb{R}}}
\newcommand{\bbR}{\R}

\newcommand{\N}{{\mathbb{N}}}
\newcommand{\bbN}{\N}
\newcommand{\Q}{{\mathbb{Q}}}
\newcommand{\bbQ}{\Q}
\newcommand{\Quot}{{\rm Quot}}
\newcommand{\calE}{{\mathcal E}}
\newcommand{\calF}{{\mathcal F}}
\newcommand{\frp}{{\mathfrak p}}

\newcommand{\st}{:}
\newcommand{\hefresh}{\setminus}
\newcommand{\cont}{\subseteq}
\newcommand{\Mat}{{\rm Mat}}
\newcommand{\GL}{{\rm GL}}
\newcommand{\isom}{\cong}

\newcommand{\res}{{\rm res}}
\newcommand{\Gam}{\Gamma}
\newcommand{\T}{{\{t\}}}
\newcommand{\rembox}{\hfill$\blacksquare$}

\newcommand{\subdemoinfo}[2]{{\noindent\sc #1. #2}}

\def\Quot{{\rm Quot}}

\def\norm{{\rm norm}}
\def\lsdp{\ltimes}

\def\dotcup{
 \def\dotcupD{\mathbin{\mathop{\smash\cup}\limits^\cdot}}
 \def\dotcupS{\mathbin{\mathop{\smash\cup}\limits^\cdot}}
 \mathchoice{\dotcupD}{\dotcupD}{\dotcupS}{}}
\def\dotunion{
\def\dotunionD{\bigcup\kern-10.5pt\cdot\kern5pt}
\def\dotunionT{\bigcup\kern-8.5pt\cdot\kern3.5pt}
\mathop{\mathchoice{\dotunionD}{\dotunionT}{}{}}}

\numberwithin{equation}{Theorem}

\begin{document}

\bibliographystyle{alpha}

\title{Split embedding problems over the\\ open arithmetic disc\thanks{This research was supported by the DFG program ``Initiation and Intensification of Bilateral Cooperation''.}}
\author{Arno Fehm\thanks{University of Konstanz}\; and Elad Paran\thanks{Open University of Israel}} \date{\today}

\maketitle

\abstract{
Let $\Z\T$ be the ring of arithmetic power series that converge on the complex open unit disc.
A classical result of Harbater asserts that every finite group occurs as a Galois group over the quotient field of $\Z\T$. 
We strengthen this by showing that every finite split embedding problem over $\Q$ acquires a solution over this field.
More generally, we solve all $t$-unramified finite split embedding problems over the quotient field of $\mathcal{O}_K\T$, where $\mathcal{O}_K$ is the ring of integers of an arbitrary number field $K$. 
}

\section{Introduction}

The inverse Galois problem (IGP) over a field $K$ asks whether all finite groups occur as Galois groups over $K$. 
By Hilbert's irreducibility theorem, a positive answer to the classical IGP over the field of rational numbers $\Q$ would follow from a positive answer to the IGP over the rational function field $\Q(t)$.

In \cite{HarbaterCAPS}, \cite{HarbaterARAPS} and \cite{HarbaterGCAS}, Harbater introduced and studied a family of fields which ``approximate" the field $\Q(t)$: Given a positive number $r$, consider the ring $\Z_r[[t]]$ 
of continuous complex valued functions on the closed disc of radius $r$ that are holomorphic on the interior and whose Taylor expansion
has only integer coefficients.
For each $r < 1$,  $\Z_r[[t]]$ properly contains the ring $\Z[t]$ of polynomials over the integers, while for each $r \geq 1$, $\Z_r[[t]]$ coincides with $\Z[t]$. Intuitively, the bigger $r (<1)$ is, the closer $\Z_r[[t]]$ is to $\Z[t]$, and hence the closer the quotient field $\Quot(\Z_r[[t]])$ is to $\Q(t)$. Harbater also considered a ``final" object of this family, closest to $\Z[t]$ - the ring $\Z\{t\}$ = $\bigcap_{r < 1} \Z_r[[t]]$ of arithmetic power series which converge on the open unit disc. 
Harbater proved that the IGP has a positive solution over the quotient field of each of the rings $\Z_r[[t]], r < 1$, and $\Z\{t\}$. 

In order to prove this result, Harbater introduced his ``patching" method, which enables one to patch realizations of cyclic groups over these fields into a realization of a group generated by them. The patching of Galois groups has since become a central method in Galois theory, leading to several major results. 
For example, the solution of Abhyankar's conjecture by Harbater and Raynaud, and Pop's 
solution over $K(X)$ of all finite split embedding problems over $K$, for any {\em ample} (or {\em large}) field $K$, cf.~\cite{Pop}.
A {\em finite split embedding problem} (FSEP) $\Gamma \lsdp G \to \Gamma$ over $K$ is given by a finite Galois extension $L$ of $K$ with Galois group $\Gamma$ acting on a finite group $G$. 
A {\em solution} to this embedding problem over a regular extension $E$ of $K$ is a Galois extension $F$ of $E$ containing $L$, such that $\Gal(F/E) \isom \Gamma \lsdp G$, and the restriction of automorphisms from $F$ to $EL$ coincides with the projection onto $\Gamma$.

The solvability of split embedding problems is a much stronger Galois theoretic property than merely the realization of finite groups.
It has been extensively studied in recent years, for example in \cite{Pop}, \cite{HarbaterStevenson}, \cite{Paran}, \cite{BSHH}, \cite{PopHenselian}.
In \cite{FehmParan} it is shown that for each $r < 1$, the quotient field $E_r$ of $\Z_r[[t]]$ is ample and hence every FSEP over $E_r$ is solvable over $E_r(X)$,
from which, using the fact that $E_r$ is Hilbertian, one can deduce the solvability of every FSEP over $E_r$ itself, thus extending Harbater's solution of the IGP over $E_r$. However, the field $E = \Quot(\bbZ\{t\})$, which is the most interesting within this family (being closest to $\Q(t)$) remained outside the scope of the results of \cite{FehmParan}: It is unknown whether $E$ is ample, see discussion after Theorem \ref{cp6.6}. Nevertheless, in this work we show that every FSEP over the rational numbers $\Q$ is solvable over $E$. More generally, we prove the following result on the solvability of split embedding problems over $E$ (which generalizes Harbater's solution of the IGP over $E$):

\begin{MainTheorem} 
Let $E'/E$ be a finite Galois extension with group $\Gamma$ acting on a finite group $G$, and suppose the prime $t$ of $E$ is unramified in $E'$. 
Then the FSEP $\Gamma \lsdp G \to \Gamma$ has a solution.
\end{MainTheorem}

For the field $E$ the methods of \cite{FehmParan} fail altogether. Therefore, here we take a different path -- we exploit the axiomatic approach to patching developed by Haran-Jarden-V\"olklein in \cite{HaV96}, \cite{HaJ98a} and \cite{HaJ98b}, where we replace the ``analytic rings" constructed there by a new type of rings of power series with special properties, generalizing the rings used by Harbater in \cite{HarbaterCAPS}. In order to deal with the Galois action defined by a given FSEP, we work with rings of power series whose coefficients belong to the ring of integers of a certain number field. This Dedekind domain usually does not have the nice properties of $\bbZ$ (most notably, it need not be factorial), which leads to more delicate number theoretic constructions than in \cite{HarbaterCAPS}, which, in particular, yield a stronger matrix factorization result. 

In the recent work \cite{Poineau}, Poineau applies patching of analytic Berkovich spaces in order to extend Harbater's solution of the IGP from $E$ to the quotient field $E_K$ of $\mathcal{O}_K\T$, where $K$ is an arbitrary number field with ring of integers $\mathcal{O}_K$. Using our constructions we are able to generalize that result as well, replacing the field $E$ in our Main Theorem with $E_K$ (Theorem \ref{cp6.6}). Moreover, we show that our solutions are regular over $K$.

\subsection*{Acknowledgements}

We thank the referee for his careful reading and many helpful suggestions.

\section{Analytic fields}

In this section we define the analytic rings needed for the patching machinery and prove some of their basic properties.

\subsection{Rings of convergent power series}

We start by defining rings of convergent power series.

\begin{Definition}
A {\bf norm} on a ring $R$ is a map $||\cdot||:R\rightarrow\bbR_{\geq0}$ such that
$||\pm1||=1$, $||x||=0$ iff $x=0$, $||x+y||\leq||x||+||y||$, and $||xy||\leq||x||\cdot||y||$ for all $x,y\in R$.
If moreover $||x+y||\leq\max\{||x||,||y||\}$, then $||\cdot||$ is called {\bf ultrametric}.
A norm is an {\bf absolute value} if it satisfies $||xy||=||x||\cdot||y||$ for all $x,y\in R$.
An absolute value that is not ultrametric is {\bf archimedean}.\rembox
\end{Definition}

\begin{Remark}\label{cp2.1} 
Fix an algebraic closure $\tilde{\bbQ}$ of $\bbQ$.
The field of complex numbers $\bbC$ is complete with respect to the usual archimedean absolute value $|\cdot|$.
The restriction of this absolute value to $\bbQ$ extends to a norm 
on $\tilde{\bbQ}$, by 
$$
 ||x|| = \max_{\sigma \in \Hom(\tilde{\bbQ},\bbC)} |\sigma(x)|,
$$
where $\Hom(\tilde{\bbQ},\bbC)$ denotes the set of all embeddings of $\tilde{\bbQ}$ into $\bbC$.
Note that if $K$ is a number field (that is, a finite extension of $\bbQ$), then 
$$
 ||x|| = \max_{\sigma \in \Hom(K,\bbC)} |\sigma(x)|
$$ 
for each $x \in K$, and $\Hom(K,\bbC)$ is the finite set of embeddings of $K$ into $\bbC$.
\rembox
\end{Remark}

\begin{Notation}\label{cp2.2} 
For each $r > 0$, let 
$\bbC_r[[t]]$ be the ring of continuous complex functions on the closed disc of radius $r$ around the origin which are holomorphic on the interior of the disc. Then $\bbC_r[[t]]$ is complete with respect to the uniform norm $|\cdot|_r$ on the closed disc of radius $r$. 
Note that
$$
 \bbC_{1^-}[[t]] := \bigcap_{r<1}\bbC_r[[t]] = \left\{ \sum_{n=0}^\infty a_nt^n\in\bbC[[t]] \;:\; \limsup_{n\rightarrow\infty}|a_n|^{1/n}\leq1\right\}
$$
is the ring of holomorphic functions on the open unit disc.
We will identify each $\sigma\in\Hom(\tilde{\bbQ},\bbC)$ with its extension $\sigma\in\Hom(\tilde{\bbQ}[[t]],\bbC[[t]])$ given by
$\sigma(\sum_{n=0}^\infty a_nt^n)=\sum_{n=0}^\infty\sigma(a_n)t^n$.
For $r > 0$ put 
\begin{eqnarray*}
 \tilde{\bbQ}_r[[t]] &:=& \bigcap_{\sigma \in \Hom(\tilde{\bbQ},\bbC)} \sigma^{-1}(\bbC_r[[t]]) \\
  &=& \left\{ f\in\tilde{\bbQ}[[t]] \;:\; \sigma(f)\in\bbC_r[[t]]  \mbox{ for all }\sigma\in\Hom(\tilde{\bbQ},\bbC)\right\}.
\end{eqnarray*}
We also put 
$$
 \tilde{\bbQ}\T := \bigcap_{r < 1}\tilde{\bbQ}_r[[t]].
$$ 
For a subring $R$ of $\tilde{\bbQ}$, put 
$$
 R_r[[t]] := \tilde{\bbQ}_r[[t]] \cap R[[t]]
$$ 
and 
$$
 R\T := \tilde{\bbQ}\T \cap R[[t]].
$$
\rembox
\end{Notation}

\begin{Remark}\label{cp2.3} 
If $R \cont \bbQ$, then 
$$
 R\T = R[[t]] \cap \bbC_{1^-}[[t]]
$$ 
is the ring 
of power series in $R[[t]]$ that converge on the open unit disc, and thus coincides with the ring denoted
by $R\{t\}$ in \cite{HarbaterCAPS}.
\rembox
\end{Remark}

\begin{Remark}\label{cp2.4}
If $R$ is a subring of a number field $K$, then the fact that $\Hom(K,\C)$ is finite implies that
$$
 R\T = \left\{ \sum_{n=0}^\infty a_nt^n\in R[[t]] \;:\; \limsup_{n\rightarrow\infty}||a_n||^{1/n}\leq1 \right\}.
$$
In particular, if 
$f = \sum_{n=0}^\infty a_n t^n \in R[[t]]$ satisfies $||a_n|| < C$ for all $n \geq 0$, for some constant $C$, then $f \in R\T$. 
\rembox
\end{Remark}

\subsection{A Weierstrass division theorem}
For the rest of this section, fix a number field $K$ and let $R=\mathcal{O}_K$ be its ring of integers.
We show that the rings of power series just defined satisfy a variant of the Weierstrass division theorem.

\begin{Lemma}\label{cp2.6} 
For each $0 \neq g \in R$ there exists a positive bound $C_g \in \bbR$ such that for each $f \in K$ there exists $h \in R$ satisfying 
$||f - gh|| < C_g$.
\end{Lemma}

\begin{proof} 
Let $b_1,\ldots,b_n \in R$ be an integral basis of $K$. In particular, $K = \sum_{i=1}^n b_i\bbQ$. Let $C_1 = \sum_i ||b_i||$. For $f_1 \in K$, write $f_1 = \sum_i \lambda_i b_i$ with $\lambda_1,\ldots,\lambda_n \in \bbQ$, take $\mu_1,\ldots,\mu_n \in \bbZ$ such that $|\lambda_i - \mu_i| \leq {1 \over 2}$ for each $1 \leq i \leq n$ and put $h = \sum_i \mu_i b_i \in R$. Then $||f_1-h|| = ||\sum_i (\lambda_i - \mu_i)b_i|| \leq \sum_i |\lambda_i - \mu_i|\cdot ||b_i|| \leq {1 \over 2} \cdot C_1< C_1$.

Now let $C_g = {C_1 \cdot ||g||}$. 
Given $f \in K$, let $f_1 = {f \over g} \in K$. By the previous paragraph there exists $h \in R$ such that $||f_1-h|| < C_1$, hence $||f-gh|| \leq ||g||\cdot||f_1 - h|| < C_g$. 
\end{proof}

We will in several places use the constant $C_1$, which is just $C_g$ for $g=1$.

Let $\calF$ be the family of non-trivial valuations on $K$ (corresponding to the maximal ideals of $R$, i.e.~the non-archimedean primes of $K$).
For each $a \in R$, 
let $\calF_a$ be the finite subfamily of valuations which are positive on $a$. For each $v \in \calF$, denote the valuation ring of $v$ (in $K$) by $R_v$.

\begin{Lemma}\label{cp2.7} 
Let $0 \neq g \in R$ and let $C_g$ be the bound given by Lemma \ref{cp2.6}. For $0 \neq a \in R$ and $f \in R[{1 \over a}]$ there exists $h \in R[{1 \over a}]$ such that $f-gh \in R$ and $||f-gh|| < C_g$.
\end{Lemma}

\begin{proof}
The strong approximation theorem \cite[Chapter 10, Theorem 4.1]{Cas86} gives an element $h \in K$ such that $v(h - {f \over g}) \geq 0 \geq v({1 \over g})$ for each $v \in \calF_a$, and $v(h) \geq 0$ for each $v \in \calF \hefresh \calF_a$. Since $R$ is integrally closed, so is $R[{1 \over a}]$. It follows that $R[{1 \over a}] = \bigcap_{v \in \calF \hefresh \calF_a} R_v$, hence $h \in R[{1 \over a}]$. Let $v \in \calF \hefresh \calF_a$. Then $v(f) \geq 0$ (since $f \in R[{1 \over a}]$), hence $v(h-{f\over g}) \geq \min (v(h), v({f \over g})) \geq v({1 \over g})$. We conclude that $v(h - {f \over g})\geq v({1 \over g})$ holds for all $v \in \calF$, hence $f-gh \in \bigcap_{v \in \calF} R_v = R$, as needed. By Lemma \ref{cp2.6}, we may subtract an element of $R$ from $h$ to assume that $||f-gh|| < C_g$.
\end{proof}

The following proposition can be viewed as a form of ``Weierstrass division".

\begin{Proposition}\label{cp2.8} 
Let $0 \neq a \in R$ and let $A  =  R\T[{1 \over a}]$. Let $D$ be either $R[{1 \over a}]\T$ or $R[{1 \over a}][[t]]$.
Then for each $0 \neq g \in D$ we have 
$$
 D  = A + g\cdot(1+tD).
$$
\end{Proposition}

\begin{proof}
Write $g = \sum_{i = m}^\infty g_i t^i$ with $g_m \neq 0$. Let $f = \sum_{i = 0}^\infty f_i t^i \in D$. 
It suffices to find $h \in D$ with constant term $1$ and such that $f - gh \in A$. Put $\hat{f} = \sum_{i = m}^\infty f_i t^{i - m}$. Then $f - t^m\hat{f} \in R[{1 \over a}][t] \cont A$. Replace $f$ with $\hat{f}$ and $g$ with $t^{-m}g$ to assume that $m = 0$. Since $a \in R[{1 \over a}]^\times$, we may multiply $g$ and $f$ with a power of $a$ to assume that $g_0 \in R$. We now recursively construct the coefficients of $h$ as follows: Put $h_0 = 1$. Suppose we have constructed $h_0,\ldots,h_{n-1}\in R[{1 \over a}]$. Let $b_n :=  f_n - \sum_{i+j = n, j \neq n}g_ih_j$. Apply Lemma \ref{cp2.7} to find $h_n \in R[{1 \over a}]$ such that $b_n - g_0h_n \in R$. If $D = R[{1 \over a}][[t]]$ we also assume (by Lemma \ref{cp2.7}) that $||b_n - g_0h_n|| < C_{g_0}$. If $D = R[{1 \over a}]\T$, we instead apply Lemma \ref{cp2.6} and subtract an element of $R$ from $h_n$ to assume that $||h_n|| < C_1$, where $C_1$ is the constant defined there.

Write $h = \sum_{i = 0}^\infty h_it^i \in R[{1 \over a}][[t]]$. Then the $n$-th coefficient of $f-gh$ is $b_n-g_0h_n \in R$, hence $f - gh \in R[[t]]$. If $D = R[{1 \over a}][[t]]$ then $h \in R[{1 \over a}][[t]] = D$ and the coefficients of $f-gh$ are bounded, hence $f-gh \in R\T \cont A$ (by Remark \ref{cp2.4}). If $D = R[{1 \over a}]\T$ then the coefficients of $h$ are bounded, hence $h \in R[{1 \over a}]\T = D$. Since also $f,g \in D$, we conclude that $f - gh \in D \cap R[[t]] = R\T \cont A$. 
\end{proof}

\begin{Corollary}\label{cp2.9} 
Let $0 \neq a \in R$ and let $D$ be either $R[{1 \over a}]\T$ or $R[{1 \over a}][[t]]$. Let $Q$ be the localization of $D$ by the multiplicative subset $R\T \hefresh \{0\}$. Then $Q$ is a field. Equivalently, for each $g \in D$ there exists $0 \neq h \in D$ such that $gh \in R\T$.
\end{Corollary}

\begin{proof} 
The equivalence of the two assertions is clear. Let $0 \neq g \in D$. By Proposition \ref{cp2.8}, there exists $h \in 1+tD, r \in R\T[{1 \over a}]$ such that $0 = r+gh$. Then $0\neq h\in D$ and $gh = -r \in R\T[{1 \over a}]$. Replace $h$ with $ha^m$ for a sufficiently large $m\in \bbN$ to get $gh \in R\T$. 
\end{proof}

\subsection{Analytic rings}
We now construct the analytic rings and fields for the patching machinery.
For the rest of this section we fix the following setup:

\begin{Setup}\label{cp2.10} 
Let $I$ be a finite index set. For each $i \in I$ let $a_i$ be a non-invertible element of $R$ such that $a_i,a_j$ are co-prime (as elements of the Dedekind domain $R$) for distinct $i,j \in I$. For each $J \cont I$, set $a_J = \prod_{j \in J} a_j$ (for $J = \emptyset$ put $a_\emptyset = 1$) and let $R_J = R[{1 \over a_J}]$. For each $i \in I$, set $a_i' =a_{I \hefresh \{i\}}, z_i = {a_i' \over a_i}, R_i = R_{I \hefresh \{i\}}$ and $R_i' = R_{\{i\}}$. 
\rembox
\end{Setup}

\begin{Lemma}\label{cp2.11} 
For each $J \cont I$ the ring $R_J$ equals $R[z_j \st j \in J]$ (the subring of $K$ generated over $R$ by all $z_j$ with $j \in J$). 
\end{Lemma}

\begin{proof}
Since $R$ is a Dedekind domain, so is its overring $R[z_j \st j\in J]$, see \cite[Proposition 2.4.7]{FJ},
which therefore equals the intersection of the valuation rings $R_v$ of $K$ lying above it.
Since the $a_i$ are coprime, $R[z_j \st j\in J] \cont R_v$ if and only if $v(a_j) = 0$ for each $j \in J$, equivalently, if and only if $v(a_J) = 0$. Thus $R[z_j \st j\in J] = \bigcap_{v \in \calF \hefresh \calF_{a_J}} R_v = R[{1 \over a_J}] = R_J$.
\end{proof}

\begin{Lemma}\label{cp2.12} 
The intersection $\bigcap_{i \in I} R_i$ equals $R$.
\end{Lemma}

\begin{proof}
Let $v \in \calF$. If there exists $i \in I$ such that $v(a_i')  > 0$ then $v(a_j) > 0$ for some $j \neq i$, hence (since $a_k$ is co-prime to $a_j$ for each $k \neq j$) we have $v(a_j') = 0$, which implies that $\bigcap_{k \in I} R_k \cont R_j = R[{1 \over a_j'}] \cont R_v$. If $v(a_i') = 0$ for all $i \in I$, then also $\bigcap_{k \in I} R_k \cont R_v$. Thus $\bigcap_{k \in I} R_k \cont \bigcap_{v \in \calF} R_v = R$. 
\end{proof}

\begin{Proposition}\label{cp2.13} 
For each $i \in I$ and $y \in R_I$ there exist $y_i \in R_i$ and $y'_i \in R_i'$ such that $y = y_i + y_i'$. 
\end{Proposition}

\begin{proof}
Let $y \in R_I$. By Lemma \ref{cp2.11} we have $R_I = R[z_j \st j \in I], R_i = R[z_j \st j \neq i]$ and $R_i' = R[z_i]$. Thus without loss of generality we may assume that $y$ is a monomial of the form $b\prod_{j \in I} z_j^{e_j}$ with $b \in R$ and $e_j$ a non-negative integer for each $j \in I$. We list the elements of $I$ as $\{i_1,\ldots,i_n\}$ such that $i = i_1$ and $e_{i_2} \leq e_{i_3} \leq \ldots \leq e_{i_n}$. Note that for distinct $j, k \in I$ we have

$$z_j \cdot z_k = {\prod_{l \in I}{a_l} \over a_j^2} \cdot {\prod_{l \in I}{a_l} \over a_k^2} = {\prod_{l \in I}a_l^2 \over a_j^2 a_k^2} = \prod_{l\neq j,k}a^2_l\in R.$$ Thus, we have $$z_{i_2}^{e_{i_2}} \cdot z_{i_3}^{e_{i_3}} = (z_{i_2} \cdot z_{i_3})^{e_{i_2}} \cdot z_{i_3}^{e_{i_3} - e_{i_2}}= c \cdot z_{i_3}^{e_{i_3} - e_{i_2}}$$ for some $c \in R$. Replace $b$ with $b\cdot c$ and $e_{i_3}$ with $e_{i_3} - e_{i_2}$ to assume that $e_{i_2} = 0$. Proceeding by induction we may assume that $e_{i_2} = e_{i_3} = \ldots = e_{i_{n-1}} = 0$. So $y = b \cdot z_{i_1}^{e_{i_1}} \cdot z_{i_n}^{e_{i_n}}$. If $e_{i_n} \geq e_{i_1}$ then $y \in R[z_{i_n}] \subseteq R[z_j \st j \neq {i_1}] = R[z_j \st j \neq {i}] = R_i$, and if $e_{i_n} < e_{i_1}$ then $y \in R[z_{i_1}] = R[z_i] = R_i'$, as needed.
\end{proof}

For the rest of this section, set $D = R\T$ and $E = \Quot(D)$. We wish to construct ``analytic fields" over $E$. 

\begin{Construction}\label{cp2.14} 
Suppose the index set $I$ contains $1$ (as a symbol) and let $J \cont I$. 
If $1 \notin J$, put $D_J = R_{J}\T$ and if $1 \in J$ put $D_J = R_J[[t]]$. 
For each $i \in I$, let $D_i = D_{I \hefresh \{i\}}$ and $D_i' = D_{\{i\}}$. We view all these rings as contained in the common ring $D_I = R_I[[t]]$. 
Note that 
$$
 D_\emptyset = R\T = D,
$$
$$
 D_1 = D_{I \hefresh \{1\}} = R_1\T,\quad D_1' = D_{\{1\}} = R_1'[[t]]
$$ 
and 
$$
 D_i = D_{I \hefresh \{i\}} = R_i[[t]],\quad D_i' = D_{\{i\}} = R_i'\T
$$ 
for each $i \neq 1$.

$$ 
\begin{xy}
\xymatrix@=10pt{& D_i \ar@{-}[rr] \ar@{-}[dd] \ar@{-}[dl] && D_I \ar@{-}[dd] \ar@{-}[dl]  \\
R_i \ar@{-}[rr] \ar@{-}[dd] && R_I  \ar@{-}[dd]\\
& D \ar@{-}[rr] \ar@{-}[dl] && D_i' \ar@{-}[dl] \\
R \ar@{-}[rr] && R_i'}
\end{xy}
$$

Let $Q_J=\Quot(D_J)$, and note that $Q_\emptyset = \Quot(R\T) = E$.
Let $Q = \Quot(D_I)$ and for each $i \in I$ put $Q_i = Q_{I\hefresh \{i\}}=\Quot(D_i)$, $Q'_i = \bigcap_{j \neq i} Q_j$. Note that $E,D'_i \subseteq Q'_i$ for each $i \in I$, since $D_i' \subseteq D_j$ for each $i \neq j \in I$.

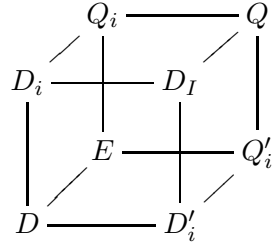
$$ 
\begin{xy}
\xymatrix@=10pt{& Q_i \ar@{-}[rr] \ar@{-}[dd] \ar@{-}[dl] && Q \ar@{-}[dd] \ar@{-}[dl] \\
D_i \ar@{-}[rr] \ar@{-}[dd] && D_I \ar@{-}[dd]\\
& E \ar@{-}[rr] \ar@{-}[dl] && Q_i' \ar@{-}[dl] \\
D \ar@{-}[rr] && D_i'}
\end{xy}
$$
\rembox
\end{Construction}

For the rest of this section we fix the notation of Construction \ref{cp2.14}.

\begin{Proposition}\label{cp2.15} 
The intersection $\bigcap_{i \in I} Q_i$ equals $E$. 
\end{Proposition}

\begin{proof} 
Let $y \in \bigcap_{i\in I} Q_i$. For each $i \in I$, by Corollary \ref{cp2.9} there exists $p_i \in R\T$ such that $p_iy \in D_i$. Put $p = \prod_{i \in I}p_i$. Then $py \in \bigcap_{i \in I} R_i[[t]]$ and $py \in D_1 \cont K\T$. It follows from Lemma \ref{cp2.12} that $py \in R\T$, hence $y = {py \over p} \in \Quot(R\T) = E$.
\end{proof}

For any domain $R$ we denote the $t$-adic absolute value on $R[[t]]$ by $|\cdot|_t$. That is, we let $|f|_t = e^{-v_t(f)}$, where $v_t$ is the (normalized) $t$-adic valuation on $R[[t]]$.
Recall that $C_1$ is the constant $C_g$ from Lemma~\ref{cp2.6} for $g=1$.

\begin{Proposition}\label{cp2.16} 
For each $f \in D_I$ and $i \in I$ there exist $g = \sum_{n = 0}^\infty g_n t^n$ $\in D_i',h = \sum_{n = 0}^\infty h_n t^n\in D_i$ such that $f = g + h$ and $|g|_t \leq |f|_t$ and $|h|_t \leq |f|_t$. Moreover, if $i = 1$ then $||h_n|| < C_1$ and if $i \neq 1$ then $||g_n|| < C_1$, for all $n \geq 0$. 
\end{Proposition}

\begin{proof}
Let $f = \sum_{n = m}^\infty f_nt^n \in D_I$ with $f_m \neq 0$. For each $n \geq m$, Proposition \ref{cp2.13} gives $g_n \in R_i', h_n \in R_i$ such that $f_n = g_n+h_n$. By Lemma \ref{cp2.6} we may add an element of $R$ to $g_n$ and subtract it from $h_n$ to assume that: If $i = 1$ then $||h_n|| < C_1$, and if $i \neq 1$ then $||g_n|| < C_1$. Then $g = \sum_{n = m}^\infty g_n t^n \in R_i'[[t]], h = \sum_{n = m}^\infty h_n \in R_i[[t]]$ satisfy $f = g+h$. Moreover, if $i = 1$, then $h \in R_i\T = D_i$, $g \in R_i'[[t]] = D_i'$, and if $i \neq 1$ then $g \in R_i'\T = D_i'$, $h \in R_i[[t]] = D_i$.
Clearly, $|g|_t \leq |f|_t$, $|h|_t \leq |f|_t$. 
\end{proof}

\begin{Construction}\label{cp2.17} 
Consider the case where $K$ is a finite Galois extension of a subfield $K_0$ with $a_1,a_I \in K_0^\times$, and suppose there is an action of $\Gamma = \Gal(K/K_0)$ on the set $I$, such that $a_i^\gamma = a_{i^\gamma}$ for all $i \in I, \gamma \in \Gamma$. Since $R$ is the ring of integers of $K$, $\Gamma$ acts on $R$. Since $a_I \in K_0$ we have $a_I^\gamma = a_I$ for all $\gamma \in \Gamma$, hence $\Gamma$ acts on $R_I = R[{1 \over a_I}]$. For each $J \cont I$ and $\gamma \in \Gamma$ we have $R_J^\gamma = R[{1 \over a_J^\gamma}] = R[{1 \over a_{J^\gamma}}] = R_{J^\gamma}$. 

The action of $\Gamma$ on $R_I$ extends to an action of $\Gamma$ on $R_I[[t]]$ (coefficient-wise). Let $J$ be a subset of $I$ and let $\gamma \in \Gamma$. If $1 \in J$, then also $1 = 1^\gamma \in J^\gamma$, hence $D_J^\gamma = R_J[[t]]^\gamma = (R_J)^\gamma[[t]] = R_{J^\gamma}[[t]] = D_{J^\gamma}$. Suppose $1 \notin J$. Then $1 = 1^\gamma \notin J^\gamma$. If $f \in D_J$, then $\sigma(f) \in \bbC_{1^-}[[t]]$ for all $\sigma \in \Hom(K,\bbC)$. 
Since $\sigma(f^\gamma)=(\sigma\circ\gamma)(f)$ and $\sigma\circ\gamma\in\Hom(K,\bbC)$, 
$f^\gamma \in R_{J^\gamma}[[t]]$ also satisfies $\sigma(f^\gamma) \in \bbC_{1^-}[[t]]$ for all $\sigma \in \Hom(K,\bbC)$. 
Thus in this case as well we have $D_J^\gamma = D_{J^\gamma}$.

Thus the action of $\Gamma$ extends to $Q = \Quot(D_I)$ (by $({f \over g})^\gamma = {f^\gamma \over g^\gamma})$, satisfying $Q_i^\gamma = Q_{i^\gamma},(Q'_i)^\gamma = Q'_{i^\gamma} $ for all $i \in I, \gamma \in \Gamma$. Note also that $E^\gamma = E$ for each $\gamma \in \Gamma$, since $R^\gamma = R$.
\rembox
\end{Construction}

\section{Matrix factorization}

We now combine the additive decomposition $D_I=D_i+D_i'$ of Proposition \ref{cp2.16} with
the ``Weierstrass division" of Proposition \ref{cp2.8} and some completeness arguments
to get a matrix factorization result for the quotient fields of our analytic rings (Corollary \ref{cp3.5}).

\begin{Remark}\label{cp3.1} 
Let $D$ be a domain, equipped with a norm $|\cdot|$. 
For each $n \geq 1$, $|\cdot|$ extends to a metric on the ring of $n\times n$-matrices $\Mat_n(D)$, by $|(a_{ij})_{1 \leq i,j \leq n}| = \max_{i,j}|a_{ij}|$, satisfying $|a+b| \leq |a|+|b|$ and $|ab| \leq n|a||b|$ (and if $|\cdot|$ is ultrametric, then $|a+b| \leq \max\{|a|,|b|\}, |ab| \leq |a||b|$).
Clearly, $D$ is complete w.r.t. $|\cdot|$ if and only if $\Mat_n(D)$ is complete for each $n \geq 1$.
\rembox
\end{Remark}

\begin{Proposition}\label{cp3.2} 
Let $A$ be a domain equipped with an ultrametric absolute value $|\cdot|$ and let $A_0$ be a dense subring, such that $A = A_0 + gA$ for each $0 \neq g \in A_0$. Put $E = \Quot(A), E_0 = \Quot(A_0)$ and let $E_1,E_2$ be subfields of $E$ such that $E_0 \cont E_2$. Let $n \in \bbN$ and suppose that:
\begin{enumerate}
\item[(i)] The localization $(A_0 \hefresh \{0\})^{-1}A$ equals $E$.
\item[(ii)] For for each $b \in \Mat_n(A)$ with $|b-\mathbbm{1}| < 1$  (where $\mathbbm{1}$ is the identity matrix of rank $n$) there exist $b_1 \in \GL_n(E_1), b_2 \in \GL_n(E_2)$ such that $b = b_1 b_2$.
\end{enumerate}
\noindent Then $\GL_n(E) = \GL_n(E_1)\cdot \GL_n(E_2)$. 
\end{Proposition}

\begin{proof}
Let $b \in \GL_n(E)$. By (i) there exists $0 \neq h \in A_0$ such that $hb \in \Mat_n(A)$. If $hb = b_1b_2'$ with $b_1 \in \GL_n(E_1)$, $b_2' \in \GL_n(E_2)$, then $b = b_1 b_2$, where $b_2 = {1 \over h}b_2' \in \GL_n(E_2)$. So we may assume that $b \in \Mat_n(A)$.
 Let $0 \neq d = \det(b) \in A$. Let $b'' \in \Mat_n(A)$ be the adjugate matrix of $b$, so that $bb'' = d\mathbbm{1}$. By (i) again, there exists $0 \neq f \in A, 0 \neq g \in A_0$ such that ${1 \over d} = {f\over g}$. Set $b' = fb''$. Then $b' \in \Mat_n(A)$ and $bb' = g\mathbbm{1}$. Put
$$
 V = \{a' \in \Mat_n(A) \st ba' \in g\Mat_n(A)\} \;\hbox{ and }\; V_0 = V \cap \Mat_n(A_0).
$$ 
Then $V$ is an additive subgroup of $\Mat_n(A)$ and $g\Mat_n(A) \leq V$. Since $A = A_0 + gA$, we also have $\Mat_n(A) =\Mat_n(A_0) +g \Mat_n(A)$, hence $V = V_0 + g\Mat_n(A)$. Since $A_0$ is dense in $A$, $\Mat_n(A_0)$ is dense in $\Mat_n(A)$, hence $g\Mat_n(A_0)$ is dense in $g\Mat_n(A)$. It follows that $V_0 = V_0 + g\Mat_n(A_0)$ is dense in $V = V_0 +g\Mat_n(A)$. Since $b' \in V$, there exists $a_0 \in V_0$ with $|b' - a_0| < {|g| \over |b|}$. In particular, $a_0 \in \Mat_n(A_0)$ and $ba_0 \in g \Mat_n(A)$.

 Put $a = {1 \over g}a_0 \in \Mat_n(E_0)$. Then $ba = {1 \over g}ba_0\in\Mat_n(A)$ and $|\mathbbm{1} - ba| = |{1 \over g}b(b'-a_0)| \leq {1 \over |g|}|b||b'-a_0| < 1$. By (ii), there exist $b_1 \in \GL_n(E_1)$, $b_2' \in \GL_n(E_2)$ such that $b a = b_1b_2'$. In particular, $\det(a) \neq 0$, hence $a \in \GL_n(E_0)$. Thus $b =b_1b_2$ for $b_2 = b_2'a^{-1} \in \GL_n(E_2) \GL_n(E_0) = \GL_n(E_2)$.
\end{proof}

\begin{Lemma}\label{cp3.3} 
Let $D$ be a domain complete with respect to a norm $|\cdot|$. Let $(a_k)_{k=1}^\infty$ be a sequence of matrices in $\Mat_n(D)$ such that $\sum_{k = 1}^\infty |a_k| < \infty$. Then the infinite products $\ldots \cdot (\mathbbm{1}+a_2) \cdot (\mathbbm{1}+a_1)$ and $(\mathbbm{1}+a_1)\cdot (\mathbbm{1}+a_2) \cdot \ldots$ both converge in $\Mat_n(D)$. Moreover, if each $\mathbbm{1}+a_k$ is invertible, so is each of the products. 
\end{Lemma}

\begin{proof}
The proof is verbally the same as the proof of \cite[Lemma 2.2]{HarbaterCAPS} (where the lemma is proven for specific complete normed domains).
\end{proof}

\begin{Remark}\label{rem:limits}
If $(f_k)_{k\in\mathbb{N}}$ is a sequence in $\bbC_r[[t]]$ that converges both with respect to $|\cdot|_t$ in $\bbC[[t]]$
and with respect to $|\cdot|_r$ in $\bbC_r[[t]]$, then the two limits coincide.
Indeed, if $g\in\bbC_r[[t]]$ with $||g-f_k||_r\rightarrow 0$, then, writing $f_k=\sum_{n=0}^\infty f_{kn}t^n$
and $g=\sum_{n=0}^\infty g_nt^n$, the Cauchy integral formula
implies that $|g_n-f_{kn}|\leq ||g-f_k||_rr^{-n}\rightarrow 0$ as $k\rightarrow\infty$.
Since $f_k$ converges also $t$-adically, the sequence $(f_{kn})_{n\in\mathbb{N}}$ is eventually constant for every $k$, and 
therefore eventually equal to $g_k$. Hence, $g$ is also the limit of $(f_k)_{k\in\mathbb{N}}$ with respect to $|\cdot|_t$.
\rembox
\end{Remark}

For the rest of this section we use the notation of Construction \ref{cp2.14}.

\begin{Proposition}\label{cp3.4} 
Let $n \in \bbN$ and let $b \in \Mat_n(D_I)$ satisfy $|b-\mathbbm{1}|_t  < 1$. Then for each $i \in I$ there exist $b_i' \in \GL_n(Q_i')$ and $b_i \in \GL_n(Q_i)$ such that $b = b_i'b_i$. 
\end{Proposition}

\begin{proof}
Let $C_1$ be given by Lemma \ref{cp2.6}. 
By Proposition \ref{cp2.16}, for each $y \in \Mat_n(D_I)$ there exist $y^+ \in \Mat_n(D_i')$ and $y^- \in \Mat_n(D_i)$ such that $y = y^+ + y^-$ and $|y^+|_t \leq |y|_t$, $|y^-|_t \leq |y|_t$. Moreover, if $i = 1$ and $\lambda$ is a coefficient of one of the entries of $y^-$, then $||\lambda|| < C_1$. Similarly, if $i \neq 1$ and $\lambda$ is a coefficient of an entry of $y^+$, then $||\lambda|| < C_1$.

Write $y_1 = b - \mathbbm{1}$ and $c = |y_1|_t < 1$. We recursively define a sequence of matrices $(y_j)_{j=1}^\infty \cont \Mat_n(D_I)$ by setting
\begin{eqnarray}
y_{j+1} = y_j^+y_j^- - y_j^+y_j - y_j y_j^- + y_j^+y_jy_j^-. \label{cf1}
\end{eqnarray}

Since $|\cdot|_t$ is ultrametric it follows by induction that
$|y_{j+1}|_t \leq |y_j|_t^2$ 
and  $|y_j^+|_t\leq |y_j|_t \leq c^{2^{j-1}} < 1$ for each $j \geq 1$. Thus $\det(\mathbbm{1}-y_j^+) \equiv \det(\mathbbm{1}) = 1 \mod (t)$, hence $\mathbbm{1} - y_j^+$ is invertible in the ring $\Mat_n(R_i'[[t]])$ (which is complete with respect to $|\cdot|_t$) for all $j$. Moreover, $\sum_j |y_j^+|_t < \infty$. By Lemma \ref{cp3.3}, $\ldots \cdot (\mathbbm{1}-y_2^+) \cdot (\mathbbm{1}-y_1^+)$ converges to a matrix $p'_i \in \GL_n(R_i'[[t]])$. Similarly, $(\mathbbm{1} - y_1^-) \cdot (\mathbbm{1}-y_2^-) \cdot \ldots$ converges to a matrix $p_i \in \GL_n(R_i[[t]])$. In particular, $\det(p_i),\det(p_i') \neq 0$. By (\ref{cf1}) we have 
$$
 \mathbbm{1}+y_{j+1} = (\mathbbm{1} - y_j^+)(\mathbbm{1}+y_j)(\mathbbm{1}-y_j^-),
$$ 
hence 
$$
 \mathbbm{1}+y_{j+1} = (\mathbbm{1}-y_j^+) \cdot \ldots \cdot (\mathbbm{1}-y_1^+) \cdot b \cdot (\mathbbm{1}-y_1^-) \cdot \ldots \cdot (\mathbbm{1}-y_j^-)
$$ for all $j$. Taking the $t$-adic limit, we get $\mathbbm{1} = p'_i\cdot b \cdot p_i$.

Suppose $i \neq 1$. Since $|y_j^+|_t \leq c^{2^{j-1}} \leq c^j$ for all $j$, it follows that each of the entries of $y_j^+$ is divisible by $t^j$. Let $0 < r < 1$ and suppose $f \in D_i'$ is one of the entries of $y_j^+$. Then we may write $f = t^j(\sum_{m =0}^\infty f_m t^m)$, hence for each $\sigma \in \Hom(K,\bbC)$ we have $|\sigma(f)|_r \leq r^j(\sum_{m = 0}^\infty C_1 r^m) = r^j\cdot{C_1 \over 1 - r}$. It follows that $|\sigma(y_j^+)|_r \leq r^j\cdot{C_1 \over 1 - r}$ (where $\sigma(y_j^+)$ is the matrix obtained by applying $\sigma$ to all of the entries). 
Thus, $\sum_{j = 1}^\infty |\sigma(y_j^+)|_r \leq \sum_{j = 1}^\infty r^j\cdot{C_1 \over 1 - r} = r\cdot{ C_1 \over (1-r)^2} < \infty$. By Lemma \ref{cp3.3}, 
$\ldots \cdot (\mathbbm{1}-\sigma(y_2^+)) \cdot (\mathbbm{1}-\sigma(y_1^+))$ converges w.r.t.~$|\cdot|_r$ to an element of $\Mat_n(\bbC_r[[t]])$, 
which by Remark \ref{rem:limits} must equal $\sigma(p_i')$.
We conclude that $p_i' \in \Mat_n(D_i')$. Note that since $i \neq 1$, $p_i \in \GL_n(R_i[[t]]) = \GL_n(D_i)$.

Similarly, if $i = 1$ we get that $p_i \in \Mat_n(D_i), p_i' \in \GL_n(D_i')$. Thus in both cases, $p_i \in \Mat_n(Q_i), p_i' \in \Mat_n(Q_i')$ and since $\det(p_i) \neq 0, \det(p_i') \neq 0$ we conclude that $p_i \in \GL_n(Q_i),  p_i'\in \GL_n(Q_i')$. Write $b_i = p_i^{-1}, b_i' = (p_i')^{-1}$. Then $b = b'_i \cdot b_i$.
\end{proof}

\begin{Corollary}\label{cp3.5} 
For each $i \in I$ and $n \in \bbN$ we have 
$$
 \GL_n(Q) = \GL_n(Q_i')\cdot \GL_n(Q_i) = \GL_n(Q_i)\cdot \GL_n(Q_i'). 
$$
\end{Corollary}

\begin{proof}
Let $i \in I$ and put $A = D_I$, $A_0 = R\T\cdot R_I$, $E = Q$, $E_1 = Q_i'$, $E_2 = Q_i$. Then $A$ is complete with respect to the ultrametric absolute value $|\cdot|_t$, and $A_0$ is $|\cdot|_t$-dense in $A$, since $A_0$ contains the ring $R_I[t]$. By Proposition \ref{cp2.8}, $A = A_0 + gA$ for each $0 \neq g \in A$. Clearly, $E_0 = \Quot(A_0) \cont E_2$. Condition (i) of Proposition \ref{cp3.2} holds by Corollary \ref{cp2.9} and condition (ii) holds by Proposition \ref{cp3.4}, thus $\GL_n(Q) = \GL_n(Q_i')\cdot \GL_n(Q_i)$. Consequently, $\GL_n(Q) =\GL_n(Q)^{-1} = \GL_n(Q_i)^{-1} \GL_n(Q_i')^{-1} =\GL_n(Q_i) \GL_n(Q_i')$. 
\end{proof}

\begin{Remark}\label{cp3.6} 
It is unknown to the authors whether the above matrix factorization results hold for the base rings. That is, whether $\GL_n(D_I) = \GL_n(D_i) \GL_n(D_i')$ for $i \in I$. In the case where $K = \bbQ$, the answer is positive, and a proof is given in \cite[Proposition 2.3]{HarbaterCAPS}. However, that proof does not generalize to arbitrary number fields (the critical point is that for $\bbZ$, the bound $C_1$ given by Lemma \ref{cp2.6} is 1, which ensures invertibility of certain convergent power series on the open unit disc). However, for our constructions to work, matrix factorization over the quotient fields is precisely the needed result.
\rembox
\end{Remark}

\section{Cyclic extensions}

As before, let $K$ be a number field with ring of integers $R=\mathcal{O}_K$, and let $n=[K:\Q]$.
In this section we construct cyclic Galois extensions of $E=\Quot(R\T)$ and embed them into
the quotient fields of our analytic rings.

\begin{Remark}\label{cp4.1} 
 We associate to $K$ its {\bf Minkowski space} $K_\bbR$, defined as follows. 
Let $\rho_1,\dots,\rho_r \in \Hom(K,\bbR)$ denote the real, and \linebreak $\sigma_1,\bar{\sigma}_1,\dots,\sigma_s,\bar{\sigma}_s \in \Hom(K,\bbC)$ the complex embeddings of $K$ into $\bbC$. Then $K_\bbR = \bbR^r \times \bbC^s$, and we have an embedding $\iota \colon K \rightarrow K_\bbR$ given by 
$$
 \iota(x) = (\rho_1(x),\dots,\rho_r(x),\sigma_1(x),\dots,\sigma_s(x))
$$ 
for each $x \in K$. Note that $n = r+2s$. We identify $K_\bbR$ with $\bbR^n$ via the standard isomorphism $\bbC\rightarrow\bbR^2$.

For any $\bbR$-basis $B=(b_1,\dots,b_n)$ of $\bbR^n$, the norm on $\bbR^n$ defined by
$$
\left\| \sum_i x_ib_i \right\|_B = \max_i|x_i|,\quad x_1,\dots,x_n\in\bbR
$$
is equivalent to the Euclidean norm $|\cdot|_{\bbR^n}$. In particular, $|\cdot|_{\bbR^n}$ is equivalent to $||\cdot||_e$, where $e$ is the standard basis of $\bbR^n$, so the norms $||\cdot||$ (see Remark \ref{cp2.1}) and
$|\cdot|_{\bbR^n}\circ \iota$ are equivalent on $K$.
\end{Remark}

\begin{Lemma}\label{cp4.2} 
If $a\neq0$ is a non-invertible element of $R$, then there exists $0 \neq b \in \Z[{1 \over a}]$ with $||b|| < 1$. 
\end{Lemma}

\begin{proof}
Since $a^{-1}$ is not integral over $\bbZ$, $\Z[a^{-1}]$ is not finitely generated as a $\bbZ$-module. Thus (in the notation of Remark \ref{cp4.1}) $\iota(\Z[a^{-1}])$ is subgroup but not a lattice of $\bbR^n$, i.e. a non-discrete subgroup \cite[Proposition I.4.2]{Neu99}. Since the norm induced on $K$ by $|\cdot|_{\bbR^n}$ is equivalent  to $||\cdot||$, the claim follows.
\end{proof}

\begin{Proposition}\label{cp4.3} 
Let $K$ be a number field with ring of integers $R=\mathcal{O}_K$.
Let $k\in\mathbb{N}$ and suppose that $K$ contains a root of unity of order $k$.
Let $a \in R\hefresh R^\times$ and $A = R[{1 \over a}]$. Put $B = A\T$, $E = \Quot(R\T)$, and  $L = \Quot(B)$.  
Then there exists a Galois extension $F$ of $E$ contained in $L$, with $\Gal(F/E) \isom \bbZ/k\bbZ$.
\end{Proposition}

\begin{proof}
Consider the polynomial $p(X) = X^k-(1-k^2t)$. We first claim it has a root in $\bbZ[[t]]$. Indeed, $p'(1) = k, p(1) = k^2t \in p'(1)^2t\bbZ[[t]]$. By the general Hensel's Lemma \cite[Theorem 7.3]{Eis95} (for the $t$-adically complete domain $\bbZ[[t]]$), there exists $f(t)  \in \bbZ[[t]]$ with $p(f(t)) = 0$. Now view $f(t)$ as a formal power series in $\bbC[[t]]$, algebraic over $\bbC[t]$. By \cite[Theorem 2.14]{Ar68}, the field of formal power series in $\bbC((t))$ which converge at some non-zero point is algebraically closed in $\bbC((t))$, in particular $f(t)$ converges at some non-zero point. Thus, there exists $r > 0$ such that $f(t)$ converges on the open disc of radius $r$ around the origin.

By Lemma \ref{cp4.2} there exists $0 \neq b \in A$ with $||b|| < 1$, hence also $|\norm_{K/\bbQ}(b)| < 1$, which implies that $b \notin R$ (otherwise $\norm_{K/\bbQ}(b) \in \bbZ$). Thus there exists a non-archimedean prime $\frp$ of $R$ such that $v_\frp(b) < 0$ (where $v_\frp$ is the corresponding discrete valuation on $K$). Let $m \in \bbN$ be a sufficiently large integer such that $||b^m|| < r$ and $v_\frp(b^m) < -v_\frp(k^2)$. Put $g(t) = f({t b^m}) \in A[[t]]$. Then for each $\lambda \in \bbC$ with $|\lambda| < 1$ we have $|\lambda \cdot \sigma(b^m)| < r$, hence $\sigma(g)(\lambda) = f(\lambda \cdot \sigma(b^m))$ converges. It follows that $g(t) \in B$. Moreover, $g(t)$ is a root of the polynomial $q(X) = X^k - (1-(k^2 \cdot b^m)t) \in E[X]$.

We next claim that $q(X)$ is irreducible over $E$. 
Let $O$ be the completion of the localization $R_\frp$. By our assumptions, ${1 \over b^m \cdot k^2} \in \frp R_\frp$, hence by \cite[Lemma 1.5(a)]{Pa10}, the element $s = {1 \over b^m \cdot k^2} - t$ is a prime element of the ring $O[[t]]$, hence defines an $s$-adic valuation $w_s$ on $\Omega = \Quot(O[[t]])$. Then $w_s(s) = 1$ and $w_s(c) = 0$ for each constant $c \in K^\times$. By the generalized Eisenstein Criterion \cite[Section 6, Theorem 2.1]{Cas86}, the polynomial $r(X) = {1 \over b^m  k^2}\cdot q(X) = {1 \over b^m k^2}\cdot X^k - ({1 \over b^m k^2} - t) \in O[[t]][X]$ is irreducible over $\Omega$, hence $q(X)$ is irreducible over $\Omega$ and hence irreducible over the subfield $E$.

Let $F$ be the splitting field of $q(X)$ over $E$. Since $K \cont E$ contains a primitive $k$-th root of unity, $F/E$ is a Kummer extension with Galois group $\bbZ/k\bbZ$. Since $q(X)$ has a root $g(t) \in L$, we have $F \cont L$. 
\end{proof}

\section{Algebraic extensions of rings of power series}

Let $R$ be an integral domain with quotient field $K$. Let $R[[t]]$ be the ring of formal power series over $R$, viewed as a subring of 
$$
 B \;:=\; \bigcup_{0 \neq a \in R} R[[{t \over a}]]
$$ 
(the union taken inside $K((t))$). Put $F = \Quot(B)$. The goal of this section is to show that $F$ is separably closed in $K((t))$. This result may be viewed as an arithmetic version of the theorem of Artin \cite[Theorem 2.14]{Ar68} concerning convergent power series (with respect to an absolute value). Our proof is an adaptation of the proof of \cite[Proposition 3.10]{Pa08}.

\begin{Lemma}\label{cp5.1} 
Let $h(Y) = p_d Y^d + \ldots + p_1 Y + p_0 \in K[[t]][Y]$ be a polynomial over
$K[[t]]$. Suppose that for each $0 \leq k \leq d$ we have $p_k =
\sum_{n=0}^\infty b_{k,n}t^n$, where $\{b_{k,n}\} \cont K$ satisfy
$b_{0,0} = 0, b_{1,0} = 1, b_{2,0} = \ldots =
b_{d,0} = 0$. Suppose $y = \sum_{n=0}^\infty a_n t^n \in K[[t]]$ is a
root of $h$. Then for each $n \geq 1$, $a_n$ is a sum of products
of the form $\pm b_{k,j_0} a_{j_1} a_{j_2} \cdot \cdot \cdot
a_{j_k}$, with $0 \leq k \leq d, \ 0 \leq j_0 \leq n, \ 0 < j_1,
\ldots, j_k < n$ such that $j_0 + j_1 + \ldots + j_k = n$.
\end{Lemma}

\begin{proof}
The proof is the same as the proof of \cite[Lemma 3.9]{Pa08}, with $F$ there replaced by $K$ here. Note that the assumption in 
\cite[Lemma 3.9]{Pa08} that almost all of the coefficients are zero is not used in the proof.
\end{proof}

\begin{Remark}\label{cp5.2} 
Let $y(t) = \sum_{n=0}^\infty a_n t^n \in K[[t]]$. If $y(at) \in B$ for some $0 \neq a \in R$, then $y(t) \in B$. Indeed, $y(at) \in B$ implies the existence of $b \in R$ such that $y(at) \in R[[{t \over b}]]$, which implies that $y(t) \in R[[{t\over ab}]] \cont B$. 
\rembox
\end{Remark}

\begin{Proposition}\label{cp5.3} 
The field $F$ is separably closed in $K((t))$. 
\end{Proposition}

\begin{proof}
 Let $y =\sum_{n=l}^{\infty}a_nt^n \in K((t))$ be separably algebraic over $F$.

\vspace{0.2cm}
\subdemoinfo{Part A}{A shift of $y$} 

\noindent
Let $y_1,y_2,\ldots,y_d$ with $y_1 = y$ be the distinct conjugates of $y$ over $F$. Then $y \in F$ if and only if $d = 1$, so suppose $d \geq 2$. Let $v$ be the $t$-adic valuation on $K((t))$.
Extend $v$ to the algebraic closure of $K((t))$ and let $r=\max\{v(y-y_i) \st i=2,\ldots,d\}$ $(\neq\infty)$, $s=r+1$.
Let $y_i^{\prime} = {t^{-s}} (y_i - \sum_{n=l}^s a_n t^n)$ for each $1\leq i \leq d$. Then $y_1^{\prime},\ldots,y_d^{\prime}$ are the distinct conjugates of $y_1^{\prime}$ over $F$. Moreover,
$v(y_1^{\prime})\geq 1$, hence for each $2 \leq i \leq d$ we have $v(y_1^{\prime}-y_i^{\prime}) = v(y_1 - y_i) -s \leq r-s = -1$, hence $v(y_i^{\prime}) \leq -1$.

Since $K(t) \cont F$ we may replace each $y_i$ with $y_i^{\prime}$ to assume that $v(y)\geq 1$ and $v(y_i) \leq -1$ for each $2 \leq i \leq d$. In particular, $y=\sum_{n=0}^{\infty} a_n t^n$ with
$a_0=0$ and $y_1, \ldots, y_d$ are the roots of an irreducible
polynomial $h(Y)=p_dY^d+\ldots+p_1Y+p_0 \in F[Y]$.

\vspace{0.2cm}
\subdemoinfo{Part B}{The values of the coefficients}

\noindent
Multiplying by the common denominator, we may assume that $p_i \in B$ for all $i$. Thus for each $0 \leq i \leq d$ there exists $0 \neq\alpha_i \in R$ such that $p_i \in R[[{t \over \alpha_i}]]$. Without loss of generality $a = \alpha_i$ is independent of $i$. By Remark \ref{cp5.2}, if $y(at) \in B$ then so is $y(t)$. Thus we may replace each $p_i(t)$ with $p_i(at)$ to assume that $p_i(t) \in R[[t]]$.

By \cite[Lemma 3.8]{Pa08} we have $e := v(p_1)< v(p_i)$ for each $i \neq 1$.
Divide the $p_i$ by $t^e$ to assume that $v(p_1) = 0 < v(p_i)$ for
each $i \neq 1$. Therefore for each $0 \leq k \leq d$ we have $p_k =
\sum_{n=0}^\infty b_{k,n} t^n$, where $b_{k,n} \in R$, $b_{1,0} \neq 0$, and for each $k \neq 1$ we have $b_{k,0} = 0$. Put $\beta=b_{1,0}$.

By Remark \ref{cp5.2} it suffices to show that $\tilde{y}(t) = y(\beta t) \in B$.
The substitution $t \mapsto \beta t$ defines an automorphism of
$K((t))$, hence the following equality follows from $h(y) = 0$:

$${p_d(\beta t) \over \beta}\tilde{y}^d+ \ldots + {p_1(\beta t)
\over \beta}\tilde{y} + {p_0(\beta t) \over \beta} = 0.$$
The coefficients in this equality are all in $R[[t]]$, in particular $${p_1(\beta t) \over \beta} = {\beta + b_{1,1}\beta t +
\ldots \over \beta} = 1 + b_{1,1} t + \ldots,$$ hence without loss
of generality we may assume that $\beta = 1$.

\vspace{0.2cm}
\subdemoinfo{Part C}{The coefficients $a_n$}

\noindent
By Lemma \ref{cp5.1}, for each $n$, $a_n$ is a sum of products of the form
$\pm b_{k,j_0} a_{j_1} a_{j_2} \cdot \cdot \cdot a_{j_k}$, with $0 \leq k \leq d, 0 \leq j_0 \leq n, 0 < j_1,
\ldots, j_k < n$ such that $j_0 + j_1 + \ldots + j_k = n$.

We now claim that $a_n \in R$ for all $n \geq 0$. Indeed, $a_0 = 0 \in R$. Assume that $a_m \in R$ for each $0 \leq
m \leq n-1$. Then, each summand in $a_n$ belongs to $R$, hence
$a_n \in R$. Thus $y(t) \in R[[t]] \cont B \cont F$. 
\end{proof}

\begin{Corollary}\label{cp5.4} 
Suppose $K_0$ is a number field and $K$ is a finite Galois extension of $K_0$. Let $R_0$ be the ring of integers of $K_0$, and let $R$ be the ring of integers of $K$. Put $E = \Quot(R\T)$ and let $F$ be a finite Galois extension of $E$. Suppose $\phi$ is a $K$-rational place of $F$ extending the place $t \mapsto 0$ of $E$, and that $\phi$ is unramified in $F/E$. Then there exists an $E$-embedding of $F$ into $\Quot(R[{1 \over a}][[t]])$, for some $0 \neq a \in R_0$. 
\end{Corollary}

\begin{proof}
Since $\phi$ is $K$-rational and unramified in $F/K(t)$, the completion of $F$ at $\phi$ is isomorphic to $K((t))$, 
and therefore we can assume that $E\subseteq F\subseteq K((t))$.
Since $E = \Quot(R\T) \cont \Quot(R[[t]])$, Proposition \ref{cp5.3} implies that $F \cont \Quot(R[[{t \over b}]])$ for some $0 \neq b \in R$. Let $0 \neq a = \norm_{K/K_0}(b) \in R_0$. Then $R[[{t \over b}]] \cont R[[{t \over a}]]$, hence $F \cont \Quot(R[[{t \over a}]]) \cont \Quot(R[{1 \over a}][[t]])$. 
\end{proof}

\section{Solution of split embedding problems}

In this section we finally prove our Main Theorem. We start by recalling a few definitions:

\begin{Definition}
Let $E_0/K_0$ be a regular field extension.
A {\bf finite split embedding problem} (FSEP) over $E_0$ is
an epimorphism $\Gamma_1\lsdp H\rightarrow\Gamma_1$,
where $\Gamma_1=\Gal(F_1/E_0)$ for some finite Galois extension $F_1$ of $E_0$,
and $\Gamma_1$ acts on the finite group $H$.
A {\bf solution} to the FSEP
$\Gamma_1\lsdp H\rightarrow\Gamma_1$
is a finite Galois extension $F$ of $E_0$ that contains $F_1$, such that
$\Gal(F/E_0)\cong\Gamma_1\lsdp H$ and the restriction of automorphisms from $F$ to $F_1$ coincides with
the projection from $\Gamma_1\lsdp H$ onto $\Gamma_1$.
Let $K$ be the algebraic closure of $E_0$ in $F_1$.
The solution $F$ is {\bf $K_0$-regular} if $K$ is algebraically closed in $F$.

Note that any FSEP $\Gal(L_0/K_0)\lsdp H\rightarrow\Gal(L_0/K_0)$ over $K_0$ gives rise to an
FSEP $\Gal(E_0L_0/E_0)\lsdp H\rightarrow\Gal(E_0L_0/E_0)$ over $E_0$, and $\Gal(E_0L_0/E_0)\cong\Gal(L_0/K_0)$.
A {\bf solution over $E_0$} to an embedding problem over $K_0$ is a solution
to the induced embedding problem over $E_0$.
\rembox
\end{Definition}

\begin{Lemma}\label{cp6.1} 
Suppose $E_0/K_0$ is regular, $E/E_0$ is a finite Galois extension, and $\Gal(E/E_0)$ acts on a finite group $G$. Let $E'$ be a finite Galois extension of $E_0$ that contains $E$, let $\res \colon \Gal(E'/E_0) \to \Gal(E/E_0)$ be the restriction map, and let $h \colon G' \to G$ be an epimorphism of groups. 
Suppose $\Gal(E'/E_0)$ acts on $G'$ such that $h(\sigma^\gamma) = h(\sigma)^{\res(\gamma)}$ for all $\gamma \in \Gal(E'/E_0), \sigma \in G'$. 
If the FSEP 
$$
 \Gal(E'/E_0) \lsdp G'\to \Gal(E'/E_0)
$$ 
has a $K_0$-regular solution, then so does 
$$
 \Gal(E/E_0) \lsdp G \to \Gal(E/E_0). 
$$
\end{Lemma}

\begin{proof}
The proof is verbally the same as the proof of \cite[Lemma 1.1]{HaJ98b} (where the lemma is proven for rational function fields, but the proof works for arbitrary regular extensions).
\end{proof}

\begin{Proposition}\label{cp6.3} 
Let $K$ be a number field and let $R=\mathcal{O}_K$ be its ring of integers. Then $R\T$ is the compositum of $\bbZ\T$ and $R$.
\end{Proposition}

\begin{proof}
Let $f=\sum_{i=0}^\infty f_it^i\in R\T$ and let $z =(z_1,\dots,z_n) \in R^n$ be an integral basis of $K$.
Let $\iota:K\rightarrow K_\mathbb{R}$ be the embedding of $K$ into its Minkowski space (Remark \ref{cp4.1}).
By \cite[Proposition I.5.2]{Neu99}, $\iota(z)=(\iota(z_1),\dots,\iota(z_n))$ is an $\bbR$-basis of $K_\bbR$. 
By Remark \ref{cp4.1}, $||\cdot||=||\cdot||_e\circ\iota$ and $||\cdot||_{\iota(z)}\circ\iota$ are equivalent on $K$, hence there exists $c\in\bbR$ such that $||\iota(x)||_{\iota(z)}\leq c\cdot||x||$ for each $x\in K$.

For each $i$, let $f_i=\sum_{k=1}^n f_{ik}z_k$, $f_{ik}\in\bbZ$.
Then $|f_{ik}|\leq||\iota(f_k)||_{\iota(z)}\leq c\cdot||f_i||$.
Thus, $\limsup_i|f_{ik}|^{1/i}\leq \limsup_i||f_i||^{1/i}\leq 1$,
so $g_k:=\sum_i f_{ij}t^i\in\bbZ\T$ for $k=1,\dots,n$,
hence $f=\sum_k g_kz_k\in\bbZ\T\cdot R$.
\end{proof}

We shall need the following technical lemmas:

\begin{Lemma}\label{cp6.4} 
Let $K_0$ be a number field, let $K$ be a finite Galois extension of $K$, let $R=\mathcal{O}_K$ be the ring of integers of $K$, and let $0 \neq b \in R$. There exists a primitive element $z \in R \hefresh R^\times$ of $K/K_0$ such that all of its conjugates over $K_0$ are pairwise co-prime. Moreover, each of the conjugates of $z$ is co-prime to $b$.
\end{Lemma}

\begin{proof}
Set $n = [K:K_0]$, let $R=\mathcal{O}_{K_0}$ be the ring of integers of $K_0$ and let $\calF_b$ be the family of primes ideals of $R$ which contain $b$. 
By \cite[Lemma 13.3.1]{FJ} there exist infinitely many prime ideals of $R_0$ which are totally split in $K/K_0$. Choose such a prime $\frp$ of $R_0$, with the additional property that all of the primes $\mathfrak{q}_1,\ldots,\mathfrak{q}_n$ of $R$ which lie over $\frp$ do not belong to the (finite) family $\calF_b$. Since the ideal class group of a number field is finite, there exists an integer $h$ such that $\mathfrak{q}_1^h$ is a principal ideal of $R$, and we denote a generator of this ideal by $z$. It follows that $z$ has $n$ distinct conjugates and that all of the conjugates are pairwise co-prime. 
In particular, $z$ is a primitive element of $K/K_0$,
and since $z$ generated a proper ideal, $z\in R\hefresh R^\times$.
Since $\mathfrak{q}_1,\ldots,\mathfrak{q}_n$ are not in $\calF_b$, each of the conjugates of $z$ is co-prime to $b$. 
\end{proof}

\begin{Lemma}\label{cp6.5} 
Let $K_0$ be a number field and $K$ a finite Galois extension of $K_0$. Let $R$ be the ring of integers of $K$ and let $0 \neq a_1 \in R$. Let $J$ be a finite index set with $1 \notin J$. Then for each $j \in J$ there exists a primitive element $a_j \in R \hefresh R^\times$ of $K/K_0$ such that the sequence $\{a_1\} \cup \{a_j^\gamma\}_{j \in J, \gamma \in \Gal(K/K_0)}$ consists of pairwise co-prime elements.
\end{Lemma}

\begin{proof}
Let $J = \{j_1,\ldots,j_m\}$. Suppose we constructed $a_{j_1},\ldots,a_{j_{k-1}}$ (for $1 \leq k \leq m$). Then write 
$
 b = a_1 \cdot \prod_{1 \leq l \leq k-1, \gamma \in \Gal(K/K_0)} a_{j_l}^\gamma
$
and apply Lemma \ref{cp6.4} to produce the primitive element $a_{j_k}$. By induction, the elements obtained $\{a_j^\gamma\}_{j \in J, \gamma \in \Gal(K/K_0)}$ are pairwise co-prime and also co-prime to $a_1$.
\end{proof}

We are now ready to formulate and prove our Main Theorem:

\begin{Definition}
Let $K_0$ be a field, and $E_0$ an extension of $K_0$ contained in $K_0((t))$.
An FSEP $\Gal(F_1/E_0)\lsdp H\rightarrow\Gal(F_1/E_0)$ is {\bf $t$-unramified}
if the $t$-adic valuation is unramified in $F_1/E_0$.
\rembox
\end{Definition}

\begin{Theorem}\label{cp6.6} 
Let $A$ be the ring of integers of a number field $K_0$ and let $E_0 = \Quot(A\T)$. 
Then every $t$-unramified finite split embedding problem over $E_0$ has a $K_0$-regular solution.
\end{Theorem}

\begin{proof}
Let $F_1/E_0$ be a finite Galois extension, such that 
the place $\phi:t\rightarrow 0$ of $E_0$ is unramified in $F_1/E_0$, 
and assume that $\Gal(F_1/E_0)$ acts on a non-trivial group $H$. 
Let $K$ be the algebraic closure of $K_0$ in $F_1$.
We wish to solve the problem $\Gal(F_1/E_0) \lsdp H \to \Gal(F_1/E_0)$ regularly over $K_0$. 

Let $R=\mathcal{O}_K$ be the ring of integers of $K$, and let $E = \Quot(R\T)$. 
First observe that $E \subseteq F_1$. Indeed, by Proposition \ref{cp6.3} we have $R\T = \bbZ\T\cdot R \cont E_0 \cdot K \cont F_1$. Moreover, this shows that $E = E_0 \cdot K$. Since $R\T \cont K[[t]]$ we have $E \cont K((t))$, hence $E/K$ is a regular extension. Similarly, $E_0/K_0$ is a regular extension.

\vspace{0.2cm}
\subdemoinfo{Part A}{Enlarging $K$}

\noindent
This part is the proof of \cite[Proposition 4.2]{HaJ98b}. 
Extend $\phi$ to $F_1$, and let $\bar{F}_1$ be the corresponding residue field. Then $\bar{F}_1$ is a finite Galois extension of $K_0$ that contains $K$. 
Let $K'=\bar{F}_1(\zeta_k)$, where $k=|H|$ and $\zeta_k$ is a primitive $k$-th root of unity.
Let $F_1' = F_1K', E' = EK'$. Then $\phi$ extends to a $K'$-rational place $\phi'$ of $F_1'$, unramified over $E'$. Moreover, $F_1'/E_0$ is a Galois extension and its Galois group acts on $H$ via the restriction $\Gal(F'_1/E_0) \to \Gal(F_1/E_0)$. By Lemma \ref{cp6.1} and Proposition \ref{cp6.3} we may replace $F_1,E, K$ with $F_1',E', K'$ to assume that $\phi$ is a $K$-rational place of $F_1$, unramified over $E$,
and that $\zeta_k\in K$. Put $\Gamma = \Gal(K/K_0)$. Since $E_0/K_0$ is regular, $E_0$ and $K$ are linearly disjoint over $K_0$, thus $\Gal(E/E_0)$ is isomorphic to $\Gamma$ by restriction of automorphisms. Put $G_1 = \Gal(F_1/E)$.

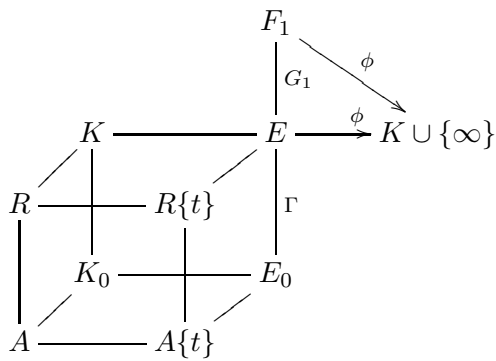
$$ 
\begin{xy}
\xymatrix@=10pt{&& && F_1 \ar@{-}[dd]^{G_1} \ar@{->}[rrdd]^\phi \\
\\
&& K \ar@{-}[rr] \ar@{-}[dl] \ar@{-}[dd] && E \ar@{-}[dl] \ar@{-}[dd]^{\Gamma} \ar@{->}[rr]^\phi && K\cup\{\infty\}\\
& R \ar@{-}[rr] \ar@{-}[dd] && R\T\ar@{-}[dd]\\
&& K_0 \ar@{-}[dl]\ar@{-}[rr] && E_0\ar@{-}[dl] \\
& A \ar@{-}[rr] && A\T }
\end{xy}
$$

\subdemoinfo{Part B}{Patching data}
\nopagebreak

\noindent
By Corollary \ref{cp5.4} we may assume that $F_1 \cont \Quot(R[{1 \over a_1}][[t]])$, for some $0 \neq a_1 \in A$. We extend $(E,F_1)$ to a patching data $\calE = (E , F_i, Q_i , Q; G_i, G )_{i \in I}$ in the sense of \cite[Definition 3.1]{HaJ98b}, as follows: Write $H$ as
\begin{eqnarray}
H=\{\tau_j\st j\in J\} \label{cf2}
\end{eqnarray}
with the index set $J$ of cardinality as of $H$ such that 
$1$ (as a symbol) does not belong to $J$.
Set $I_2=J\times\Gamma$ and let $\Gamma$ act on $I_2$ by
$(j,\gamma')^\gamma=(j,\gamma'\gamma)$. Identify $(j,1)\in I_2$
with $j$, for each $j\in J$. Then
every $i\in I_2$ can be uniquely written as
\begin{eqnarray}
\label{cf3} \mbox{  $i=j^\gamma$ with $j \in J$ and $\gamma\in \Gamma$. }
\end{eqnarray}

\noindent Let $I=\{1\}\dotcup I_2$ and extend the action of
$\Gamma$ on $I_2$ to $I$ by $1^\gamma = 1$ for each
$\gamma\in\Gamma$.

By Lemma \ref{cp6.5}, for each $j \in J$ we find a primitive element $c_j \in R \hefresh R^\times$ of $K/K_0$, such that the elements $\{c_j^\gamma\}_{j \in J, \gamma \in \Gamma}$ are pairwise co-prime and also co-prime to $a_1$. 
For each $i \in I\hefresh \{1\}$, we define an element $a_i \in R$ by writing $i = j^\gamma$ for some $j \in J, \gamma \in \Gamma$ and letting $a_i = c_j^\gamma$. Then $a_i^\gamma = a_{i^\gamma}$ for each $i \in I, \gamma \in \Gamma$ (note that this also includes the case of $i = 1$). We now use the analytic rings given by Setup \ref{cp2.10} and Construction \ref{cp2.14}. 
Note that (using the notations of Setup \ref{cp2.10}) $a_I = \prod_{i \in I}a_i = a_1 \cdot \prod_{j \in J} \norm_{K/K_0}(a_j) \in K_0^\times$. 
Extend the action of $\Gam$ to $Q$ by Construction \ref{cp2.17}. Then $Q_i^\gamma=Q_{i^\gamma}$ and $(Q'_i)^\gamma=Q'_{i^\gamma}$ for all $i\in I$ and $\gamma \in \Gam$.

\vspace{0.2cm}
\subdemoinfo{Part C}{Groups}

\noindent
This part repeats Part D of the proof of \cite[Proposition 4.1]{HaJ98b}.

Since $F_1\subseteq Q$ is a Galois extension of $E_0$, it is
$\Gamma$-invariant.
The action of $\Gamma$ on $K$ is faithful, hence it is faithful on each of the fields $E\subseteq F_1\subseteq Q$.
Thus we may identify $\Gamma$ with its image in $\Gal(F_1/E_0)$.
Then $\Gal(F_1/E_0)=\Gam\lsdp G_1$, where $\Gamma$ acts on $G_1$
by conjugation in $\Gal(F_1/E_0)$.
Thus
\begin{eqnarray}
\label{cf4}  \mbox{ $(a^\tau)^\gamma=(a^\gamma)^{\tau^\gamma}$ for all $a\in F_1$, $\tau\in G_1$, and $\gamma\in\Gamma$.}
\end{eqnarray}
The given action of $\Gal(F_1/E_0)$ on $H$
restricts to actions of its subgroups $G_1$ and $\Gamma$ on $H$.
Let $G=G_1\lsdp H$ with respect to this action,
and let $\Gamma$ act on $G$ componentwise by its action on $G_1$ and $H$.
Then
$$
\Gal(F_1/E_0)\lsdp H
=(\Gam\lsdp G_1)\lsdp H
=\Gam\lsdp(G_1\lsdp H)
=\Gam\lsdp G.
$$
Let $i \in I_2$. Use (\ref{cf3}) to write $i = j^{\gamma'}$ with unique
$j \in J$ and $\gamma' \in \Gamma$. Then define $\tau_i =
\tau_j^{\gamma'} \in H$ and observe that 
\begin{eqnarray}
\label{cf5a}
\mbox{$\tau_i^\gamma=\tau_{i^\gamma}$ for all $i\in I_2$ and $\gamma\in\Gamma$.}
\end{eqnarray}
By (\ref{cf2}), 
\begin{eqnarray}
\label{cf5b} H = \langle \tau_i \st i \in I_2 \rangle.
\end{eqnarray}

\noindent
For each $i \in I_2$ let $G_i = \langle \tau_i \rangle \le H$.
Then,
\begin{eqnarray}
\label{cf5c} &&G=\langle G_i\st i\in I\rangle \mbox{ and }H=\langle G_i\st i\in I_2\rangle; \\
\label{cf5d} &&G_i^\gamma=G_{i^\gamma}\mbox{ for all }i\in I\mbox{ and }\gamma\in\Gamma;\\
\label{cf5e} &&|I|\ge 2.
\end{eqnarray}

\subdemoinfo{Part D}{Solution of $\Gal(F_1/E_0) \lsdp H \to \Gal(F_1/E_0)$} 

\noindent
Let $j\in J$.
Since by the reductions made in {\sc Part A}, $K$ contains a primitive $k$-th root of unity, and $|G_j|$ divides $k=|H|$,
Proposition \ref{cp4.3} gives a cyclic extension $F_j/E$ with group $G_j=\langle\tau_j\rangle$, such that $F_j \cont \Quot(R_j'\T) = \Quot(D_j') \cont Q_j'\subseteq Q$.

For an arbitrary $i\in I_2$ there exist unique $j\in J$ and
$\gamma\in\Gam$ such that $i=j^\gamma$ (by (\ref{cf3})). Let
$F_i=F_j^\gamma$. Since $\gamma$ acts on $Q$ and leaves $E$ invariant, $F_i$ is a Galois
extension of $E$ and $F_i\cont Q'_i$. The isomorphism $\gamma\colon
F_j\to F_i$ gives an isomorphism $\Gal(F_j/E)\isom\Gal(F_i/E)$
which maps each $\tau\in\Gal(F_j/E)$ onto
$\gamma\circ\tau\circ\gamma^{-1}\in\Gal(F_i/E)$. In particular, it
maps $\tau_j$ onto $\gamma\circ\tau_j\circ\gamma^{-1}$. We may
therefore identify $G_i$ with $\Gal(F_i/E)$ such that $\tau_i$
coincides with $\gamma\circ\tau_j\circ\gamma^{-1}$. This means that
$(y^\tau)^\gamma=(y^\gamma)^{\tau^\gamma}$ for all $y\in F_j$ and
$\tau\in G_j$.

By Proposition \ref{cp2.15} we have $\bigcap_{i \in I}Q_i = E$, and by Corollary \ref{cp3.5} for $n = |G|$ we have $\GL_n(Q) = \GL_n(Q_i)\cdot\GL_n(Q_i')$ for each $i \in I$. By (\ref{cf5c}) we have $G = \langle G_i \st i\in I \rangle$. Thus $\calE=(E,F_i,Q_i,Q;G_i,G)_{i\in I}$ is a patching data in the sense of \cite[Definition 3.1]{HaJ98b}.

Since $F_1$ is invariant under $\Gamma$, it follows that for all $i\in I$
and $\gamma\in\Gam$ we have $F_i^\gamma=F_{i^\gamma}$. Moreover,
$(a^\tau)^\gamma=(a^\gamma)^{\tau^\gamma}$ for all $a\in F_i$ and
$\tau\in G_i$. This generalizes (\ref{cf4}).

We have shown that $\Gam$ acts properly on $\calE$, in the sense of \cite[Definition 3.1]{HaJ98b}. By \cite[Corollary 3.4(e)]{HaJ98b} the ``compound" $F\subseteq Q$ of $\calE$ is a solution for the split embedding problem $\Gal(F_1/E_0) \lsdp H \to \Gal(F_1/E_0)$.
Since $K\subseteq F\subseteq Q\subseteq K((t))$ and $K((t))/K$ is regular, so is $F/K$.
\end{proof}

The field $E =\Quot(\Z\T)$ is Hilbertian by \cite[Proposition 4.3]{FPKlein}. 
In \cite{DD} it is conjectured that every FSEP over a Hilbertian field is
solvable. This has been proven for function fields over ample fields in
\cite{Pop,HaJ98b}. The field $E$ above can be viewed as an analogue of the
(non-ample) field $K((x))(t)$ of rational functions over the complete (and hence
ample) field $K((x))$ (where $K$ is some field), see discussion in \cite[p.~855]{HarbaterGCAS}.
However, it is unclear from this analogy how to prove the non-ampleness
of $E$. Moreover, unlike the function field case where all geometric
primes play an equivalent role, in $E$ the prime $t\mapsto 0$ is distinguished,
which leads to the ramification constraint of Theorem \ref{cp6.6}. More
precisely, in the function field case, if $t\mapsto 0$ is ramified then one
chooses a non-ramified point $a \in K$ and applies the automorphism $t\mapsto t+a$
of $K[t]$ to assume that $t\mapsto 0$ is unramified (see \cite[Proof of Proposition 4.2(a)]{HaJ98b}). 
In contrast, the map $t\mapsto t+a$ does not extend to an
automorphism of $\Z\T$ (and not even to a $\C$-automorphism of $\C_{1^-}[[t]]$) for any $a \neq 0$. Nevertheless, Theorem \ref{cp6.6}
gives our main objective:

\begin{Corollary}
If $K$ is a number field with ring of integers $\mathcal{O}_K$, 
then every finite split embedding problem over $K$ has a $K$-regular solution over $\Quot(\mathcal{O}_K\T)$.
\end{Corollary}

\begin{proof}
Let $E=\Quot(\mathcal{O}_K\T)$ and let $\Gal(L/K)\lsdp H\rightarrow\Gal(L/K)$ be an FSEP over $K$.
Since $E/K$ is regular and $L/K$ is algebraic, the $t$-adic valuation is totally inert in $EL/E$,
so that the induced embedding problem $\Gal(EL/E)\lsdp H\rightarrow\Gal(EL/E)$ over $E$ is $t$-unramified.
Thus the claim follows from Theorem \ref{cp6.6}.
\end{proof}

\setlength{\parskip}{0cm}


\end{document}